\documentclass[11pt]{article}

\usepackage{amsmath}
\usepackage{amsthm}
\usepackage{amsfonts}
\usepackage{bbm}
\usepackage{amssymb}
\usepackage{graphicx}
\usepackage{color}

\usepackage{blindtext}
\usepackage{enumitem}

\usepackage{fancybox}
\newlength{\querylen}
\newcommand{\mmp}{\mathbb{P}}

\newcommand{\me}{\mathbb{E}}
\newcommand{\E}{\mathbb{E}}

\newcommand{\mr}{\mathbb{R}}
\newcommand{\mn}{\mathbb{N}}

\DeclareMathOperator{\1}{\mathbbm{1}}

\newtheorem{thm}{Theorem}[section]
\newtheorem{lemma}[thm]{Lemma}

\newtheorem{cor}[thm]{Corollary}

\newtheorem{assertion}[thm]{Proposition}
\theoremstyle{definition}

\theoremstyle{remark}
\newtheorem{rem}[thm]{Remark}

\begin{document}
\title{Renewal theory for iterated perturbed random walks on a general branching process tree: early generations}

\author{Alexander Iksanov\footnote{Faculty of Computer Science and
Cybernetics, Taras Shevchenko National University of Kyiv, Kyiv, Ukraine; \ e-mail: iksan@univ.kiev.ua} \ \ \ Bohdan Rashytov\footnote{Faculty of Computer Science and
Cybernetics, Taras Shevchenko National University of Kyiv, Kyiv, Ukraine; \ e-mail: mr.rashytov@gmail.com} \ \ and \ \ Igor Samoilenko\footnote{Faculty of Computer Science and
Cybernetics, Taras Shevchenko National University of Kyiv, Kyiv, Ukraine; \ e-mail: isamoil@i.ua}}

\maketitle
\begin{abstract}
\noindent Let $(\xi_k,\eta_k)_{k\in\mn}$ be independent identically distributed random vectors with arbitrarily dependent positive components. We call a (globally) perturbed random walk a random sequence $T:=(T_k)_{k\in\mn}$ defined by $T_k:=\xi_1+\cdots+\xi_{k-1}+\eta_k$ for $k\in\mn$. Consider a general branching process generated by $T$ and denote by $N_j(t)$ the number of the $j$th generation individuals with birth times $\leq t$. We treat early generations, that is, fixed generations $j$ which do not depend on $t$. In this setting we prove counterparts for $\mathbb{E}N_j$ of the Blackwell theorem and the key renewal theorem, prove a strong law of large numbers for $N_j$, find the first-order asymptotics for the variance of $N_j$. Also, we prove a functional limit theorem for the vector-valued process $(N_1(ut),\ldots, N_j(ut))_{u\geq 0}$, properly normalized and centered, as $t\to\infty$. The limit is a vector-valued Gaussian process whose components are integrated Brownian motions.
\end{abstract}

\noindent Key words: functional limit theorem; general branching process; key renewal theorem; perturbed random walk; renewal theory; strong law of large numbers

\noindent 2000 Mathematics Subject Classification: Primary: 60K05, 60J80
\section{Introduction}\label{Sect1}

Let $(\xi_i, \eta_i)_{i\in\mn}$ be independent copies of a $\mr^2$-valued random vector $(\xi, \eta)$ with arbitrarily dependent components. Denote by $(S_i)_{i\in\mn_0}$ ($\mn_0:=\mn\cup\{0\}$) the zero-delayed
standard random walk with increments $\xi_i$ for $i\in\mn$, that is, $S_0:=0$ and $S_i:=\xi_1+\ldots+\xi_i$ for $i\in\mn$. Define
\begin{equation*}%\label{perturbed}
T_i:=S_{i-1}+\eta_i,\quad i\in \mn.
\end{equation*}
The sequence $T:=(T_i)_{i\in\mn}$ is called {\it perturbed random walk}.

In the following we assume that $\xi$ and $\eta$ are almost surely (a.s.) positive. Now we define a general branching process generated by $T$. At time $0$ there is one individual, the ancestor. The ancestor produces offspring (the first generation) with birth times given by the points of $T$. The first generation produces the second generation. The shifts of birth times of the second generation individuals with respect to their mothers' birth times are distributed according to copies of $T$, and for different mothers these copies are independent. The second generation produces the
third one, and so on. All individuals act independently of each other. For $t\geq 0$ and $j\in\mn$, denote by $N_j(t)$ the number of the $j$th generation individuals with birth times $\leq t$ and put $V_j(t):=\me N_j(t)$, and $V_j(t):=0$ for $t<0$. Then % $N_1(t)=N(t)$, $V_1(t)=V(t)$ and
\begin{equation}\label{basic1232}
N_j(t)=\sum_{r\geq 1}N^{(r)}_{j-1}(t-T_r)\1_{\{T_r\leq t\}}=\sum_{k\geq 1}N^{(k)}_{1,\,j}(t-T^{(j-1)}_k)\1_{\{T^{(j-1)}_k\leq t\}} ,\quad j\geq 2,\quad t\geq 0,
\end{equation}
where $N_{j-1}^{(r)}(t)$ is the number of successors in the $j$th generation with birth times within $[T_r,t+T_r]$ of the first generation individual with birth time $T_r$; $T^{(j-1)}:=(T^{(j-1)}_k)_{k\geq 1}$ is some enumeration of the birth times in the $(j-1)$th generation; $N_{1,j}^{(k)}(t)$ is the number of children in the $j$th generation with birth times within $[T^{(j-1)}_k,t+T^{(j-1)}_k]$ of the $(j-1)$th generation individual with birth time $T^{(j-1)}_k$. By the branching property, $(N_{j-1}^{(1)}(t))_{t\geq 0}$, $(N_{j-1}^{(2)}(t))_{t\geq 0},\ldots$ are independent copies of $N_{j-1}$ which are also independent of $T$, and $(N_{1,\,j}^{(1)}(t))_{t\geq 0}$, $(N_{1,j}^{(2)}(t))_{t\geq 0},\ldots$ are independent copies of $(N(t))_{t\geq 0}$ which are also independent of $T^{(j-1)}$. In what follows we write $N$ for $N_1$ and $V$ for $V_1$. Passing in \eqref{basic1232} to expectations we infer, for $j\geq 2$ and $t\geq 0$,
\begin{equation}\label{convol}
V_j(t)=V^{\ast(j)}(t)=(V_{j-1}\ast V)(t)=\int_{[0,\,t]} V_{j-1}(t-y){\rm d}V(y)=\int_{[0,\,t]} V(t-y){\rm d}V_{j-1}(y),
\end{equation}
where $V^{\ast(j)}$ is the $j$-fold convolution of $V$ with itself. We call the sequence $\mathcal{T}:=(T^{(j)})_{j\in\mn}$ {\it iterated perturbed random walk on a general branching process tree}.

Following \cite{Bohun etal:2021}, we call the $j$th generation {\it early}, {\it intermediate} or {\it late} depending on whether $j$ is fixed, $j=j(t)\to\infty$ and $j(t)=o(t)$ as $t\to\infty$, or $j=j(t)$ is of order $t$. In the paper \cite{Bohun etal:2021}, to which we refer for the motivation behind the study of $\mathcal{T}$, the authors prove counterparts of the elementary renewal theorem, the Blackwell theorem and the key renewal theorem for some intermediate generations. In the present work we investigate early generations. Although the analysis of early generations is simpler than that of intermediate generations, we solve here a larger collection of problems. More precisely, we prove a strong law of large numbers for $N_j(t)$ and a functional limit theorem for the vector-valued process $(N_1(ut),N_2(ut),\ldots, N_k(ut))_{u\geq 0}$ for each $k\in\mn$, properly normalized and centered as $t\to\infty$; investigate the rate of convergence in a counterpart for $V_j$ of the elementary renewal theorem and find the asymptotics of the variance ${\rm Var}\,N_j(t)$. Also, counterparts for $V_j$ of the Blackwell theorem and the key renewal theorem are given.

The remainder of the paper is structured as follows. In Section \ref{res} we state our main results. Some auxiliary statements are discussed in Section \ref{aux21}. The proofs of the main results  are given in Section \ref{pro}. Finally, the Appendix collects a couple of assertions borrowed from other articles.

\section{Results}\label{res}

\subsection{A counterpart of the elementary renewal theorem and the rate of convergence result}

Proposition \ref{elem} is a counterpart for $V_j=\me N_j$ of the elementary renewal theorem.
\begin{assertion}\label{elem}
Assume that ${\tt m}=\me\xi<\infty$.
Then, for fixed $j\in\mn$,
\begin{equation}\label{elem1}
\lim_{t\to\infty}\frac{V_j(t)}{t^j}=\frac{1}{j!{\tt m}^j}.
\end{equation}
\end{assertion}

Theorem \ref{rate} quantifies the rate of convergence in \eqref{elem1} under the assumptions $\me\xi^2<\infty$ and $\me\eta<\infty$. As usual, $f(t)\sim g(t)$ as $t\to\infty$ means that $\lim_{t\to\infty}(f(t)/g(t))=1$.
\begin{thm}\label{rate}
Assume that the distribution of $\xi$ is nonlattice, and that $\me\xi^2<\infty$ and $\me\eta<\infty$. Then, for any fixed $j\in\mn$,
\begin{equation}\label{asymp1}
V_j(t)-\frac{t^j}{j!{\tt m}^j}~\sim~\frac{b_Vjt^{j-1}}{(j-1)!{\tt m}^{j-1}},\quad t\to\infty,
\end{equation}
where ${\tt m}=\me\xi<\infty$ and $b_V:={\tt m}^{-1}(\me \xi^2/(2{\tt m})-\me \eta)\in\mr$.
\end{thm}

\subsection{Counterparts of the key renewal theorem and the Blackwell theorem}

Theorem \ref{keyren} is a counterpart for $V_j$ of the key renewal theorem. Recall that the distribution of a positive random variable is called {\it nonlattice} if it is not concentrated on any lattice $(nd)_{n\in\mn_0}$, $d>0$.
\begin{thm}\label{keyren}
Let $f: [0,\infty)\to [0,\infty)$ be a directly Riemann integrable function on $[0,\infty)$. Assume that the distribution of $\xi$ is nonlattice and that ${\tt m}<\infty$. Then, for fixed $j\in\mn$,
\begin{equation}\label{equival}
\int_{[0,\,t]}f(t-y){\rm d}V_j(y)~\sim~ \Big(\frac{1}{{\tt m}}\int_0^\infty f(y){\rm d}y\Big) V_{j-1}(t)~\sim~ \int_0^\infty f(y){\rm d}y\frac{t^{j-1}}{(j-1)!{\tt m}^j},\quad t\to\infty.
\end{equation}
\end{thm}
Theorem \ref{blackwell_early}, a counterpart for $V_j$ of the Blackwell theorem, is just a specialization of Theorem \ref{keyren} with $f(y)=\1_{[0,\,h)}(y)$ for $y\geq 0$. Nevertheless,
we find it instructive to provide an alternative proof. The reason is that the proof given in Section \ref{pro1} illustrates nicely basic concepts of the renewal theory and may be adapted to other settings.
\begin{thm}\label{blackwell_early}
Assume that the distribution of $\xi$ is nonlattice and ${\tt m}=\me\xi<\infty$. Then, for fixed $j\in\mn$ and fixed $h>0$,
\begin{equation}\label{bla_early}
V_j(t+h)-V_j(t)~\sim~\frac{ht^{j-1}}{(j-1)!{\tt m}^j},\quad t\to\infty.
\end{equation}
\end{thm}
\begin{rem}
In the case $\eta=0$ a.s.\ limit relation \eqref{bla_early} can be found in Theorem 1.16 of \cite{Mitov+Omey:2014}.
\end{rem}

\subsection{Asymptotics of the variance}

In this section we find, for fixed $j\in\mn$, the asymptotics of ${\rm Var}\,N_j(t)$ as $t\to\infty$ under the assumption $\eta=\xi$ a.s., so that $T_k=S_k$ for $k\in\mn$. In other words, below we treat {\it iterated standard random walks}. Theorem \ref{thm:variance} is a strengthening of Lemma 4.2 in \cite{Iksanov+Kabluchko:2018} in which the big O estimate for ${\rm Var}\,N_j(t)$ was obtained, rather than precise asymptotics. We do not know the asymptotic behavior of ${\rm Var}\,N_j(t)$ for (genuine) iterated perturbed random walks.
\begin{thm}\label{thm:variance}
Assume that $\eta=\xi$ a.s., that the distribution of $\xi$ is nonlattice and ${\tt s}^2:={\rm Var}\,\xi\in (0, \infty)$. Then, for any $j\in\mn$,
\begin{equation}\label{ineq3}
\lim_{t\to\infty} \frac{{\rm Var}\,N_j(t)}{t^{2j-1}}=\frac{{\tt s}^2}{(2j-1)((j-1)!)^2 {\tt m}^{2j+1}},
\end{equation}
where ${\tt m}=\me\xi<\infty$.
\end{thm}
\begin{rem}
In view of Corollary \ref{cor onedim} the result of Theorem \ref{thm:variance} is expected, yet the complete proof requires some efforts. Of course, the relation $${\lim\inf}_{t\to\infty} \frac{{\rm Var}\,N_j(t)}{t^{2j-1}}\geq \frac{{\tt s}^2}{(2j-1)((j-1)!)^2{\tt m}^{2j+1}}$$ is an immediate consequence of \eqref{eq:onedim} and Fatou's lemma.
\end{rem}

\subsection{Strong law of large numbers}

Theorem \ref{strong} is a strong law of large numbers for $N_j$.
\begin{thm}\label{strong}
Assume that ${\tt m}=\me\xi<\infty$.
Then, for fixed $j\in\mn$,
\begin{equation}\label{slln}
\lim_{t\to\infty}\frac{N_j(t)}{t^j}=\frac{1}{{\tt m}^j j!}\quad\text{{\rm a.s.}}
\end{equation}
\end{thm}

\subsection{A functional limit theorem}

Let $B:=(B(s))_{s\geq 0}$ be a standard Brownian
motion and, for $q\geq 0$, let
$$B_q(s):=\int_{[0,\,s]}(s-y)^q {\rm d}B(y), ~~s\geq 0.$$
The process $B_q:=(B_q(s))_{s\geq 0}$ is a centered Gaussian
process called the {\it fractionally integrated Brownian motion} or the
{\it Riemann-Liouville process}. Plainly, $B=B_0$, $B_1(s)=\int_0^s B(y){\rm d}y$, $s\geq 0$ and, for integer $q\geq 2$,
$$B_q(s)=q!\int_0^s\int_0^{s_2}\ldots\int_0^{s_q} B(y){\rm d}y{\rm d}s_q\ldots{\rm d}s_2,\quad s\geq 0.$$

In the following we write $\Rightarrow$ and ${\overset{{\rm d}}\longrightarrow}$ to denote weak convergence in a function space and weak convergence
of one-dimensional distributions, respectively. As usual, we denote by $D$ the Skorokhod space of right-continuous functions defined on $[0,\infty)$ with finite limits from the left at positive points. We prefer to use $(X_t(u))_{u\geq 0} \Rightarrow (X(u))_{u\geq 0}$ in place of the formally correct notation $X_t(\cdot)\Rightarrow X(\cdot)$.

Given next is a functional limit theorem for $(N_1(ut), N_2(ut),\ldots)_{u\geq
0}$, properly normalized and centered, as $t\to\infty$.
\begin{thm}\label{func2}
Assume that ${\tt m}=\me\xi<\infty$, ${\tt s}^2={\rm Var}\, \xi\in (0,\infty)$ and $\me\eta^a<\infty$ for some $a>0$. Then
\begin{equation} \label{relFunc2}
 \bigg((j-1)! \Big(\frac{N_j(ut)-V_j(ut)}{\sqrt{{\tt m}^{-2j-1}
{\tt s}^2 t^{2j-1}}}\Big)_{u\geq 0} \bigg)_{j\in\mn}~\Rightarrow~
\big((B_{j-1}(u))_{u\geq 0}\big)_{j\in\mn},\quad t\to\infty
\end{equation}
in the product $J_1$-topology on $D^\mn$.

If $\me \eta^{1/2}<\infty$, then the centering $V_j(ut)$ can be replaced with $(ut)^j/(j! {\tt m}^j)$. If $\me\eta^{1/2}=\infty$, the centering $V_j(ut)$ can be replaced with
\begin{equation}\label{center}
\frac{\me (ut-(\eta_1+\ldots+\eta_j))^j\1_{\{\eta_1+\ldots+\eta_j\leq ut \}}}{j! {\tt m}^j}=\frac{\int_0^{t_1}\int_0^{t_2}\ldots\int_0^{t_j}\mmp\{\eta_1+\ldots+\eta_j\leq y\}{\rm d}y{\rm d}t_j\ldots{\rm d}t_2}{{\tt m}^j},
\end{equation}
where $t_1=ut$.
\end{thm}

Now we derive a one-dimensional central limit theorem for $N_j$. To this end, it is enough to restrict attention to just one coordinate in \eqref{relFunc2}, put $u=1$ there and note that $B_{j-1}(1)$ has the same distribution as $(2j-1)^{-1/2}B(1)$.
\begin{cor}\label{cor onedim}
Under the assumptions of Theorem \ref{func2}, for fixed $j\in\mn$,
\begin{equation}\label{eq:onedim}
\frac{(j-1)!(2j-1)^{1/2}{\tt m}^{j+1/2}}{{\tt s}t^{j-1/2}}\big(N_j(t)-V_j(t)\big)~{\overset{{\rm d}}\longrightarrow}~
B(1),\quad t\to\infty.
\end{equation}
\end{cor}

\section{Auxiliary tools}\label{aux21}

For $t\geq 0$, put $U(t):=\sum_{n\geq 0}\mmp\{S_n\leq t\}$, so that $U$ is the renewal function of $(S_n)_{n\in \mn_0}$. Whenever $\E\xi^2<\infty$, we have
\begin{equation}\label{lord}
0\leq U(t)-{\tt m}^{-1}t \leq c_U,\quad t\geq 0,
\end{equation}
where $c_U:={\tt m}^{-2} \me\xi^2$. The right-hand side is called {\it Lorden's inequality}. Perhaps, it is not commonly known that Lorden's inequality takes the same form for nonlattice and lattice distributions of $\xi$, and we refer to Section 3 in \cite{Bohun etal:2021} for the explanation of this fact. The left-hand inequality in \eqref{lord} is a consequence of Wald's identity $t\leq \E S_{\nu(t)}={\tt m} U(t)$, where $\nu(t):=\inf\{k\in\mn: S_k>t\}$ for $t\geq 0$.

Let us show that the left-hand inequality in \eqref{lord} extends to the convolution powers $U_j=U^{\ast(j)}$ (see \eqref{convol} for the definition) in the following sense.
\begin{lemma}\label{lefthand}
Let $j\in\mn$ and ${\tt m}=\me\xi\in (0,\infty)$. Then
\begin{equation}\label{lordleft}
U_j(t)\geq \frac{t^j}{j!{\tt m}^j},\quad t\geq 0.
\end{equation}
\end{lemma}
\begin{proof}
We use the mathematical induction. When $j=1$, \eqref{lordleft} reduces to the left-hand inequality in \eqref{lord}. Assuming that \eqref{lordleft} holds for $j\leq k$ we infer
$$U_{k+1}(t)-\frac{t^{k+1}}{(k+1)!{\tt m}^{k+1}}=\int_{[0,\,t]}\Big(U(t-y)-\frac{t-y}{{\tt m}}\Big){\rm d}U_k(y)+\frac{1}{{\tt m}}\int_0^t \Big(U_k(y)-\frac{y^k}{k!{\tt m}^k}\Big){\rm d}y\geq 0,$$ that is, \eqref{lordleft} holds with $j=k+1$.
\end{proof}

Put $G(t):=\mmp\{\eta\leq t\}$ for $t\in\mr$. Observe that $G(t)=0$ for $t<0$. Since $V(t)\leq U(t)$ for $t\geq 0$ we infer
\begin{equation}\label{lord1}
V(t)-{\tt m}^{-1}t \leq c_U,\quad t\geq 0
\end{equation}
for any distribution of $\eta$. On the other hand, assuming that $\me\eta<\infty$,
\begin{eqnarray*}
V(t)-{\tt m}^{-1}t&=&\int_{[0,\,t]}(U(t-y)-{\tt m}^{-1}(t-y)){\rm d}G(y)\\&\hphantom{==}-& {\tt m}^{-1} \int_0^t (1-G(y)){\rm d}y\geq-{\tt m}^{-1}\int_0^t (1-G(y)){\rm d}y\geq
-{\tt m}^{-1}\me\eta
\end{eqnarray*}
having utilized $U(t)\geq {\tt m}^{-1}t$ for $t\geq 0$. Assuming that $\me\eta^a<\infty$ for some $a\in (0,1)$ which particularly implies that $\lim_{t\to\infty}t^a\mmp\{\eta>t\}=0$ we infer
\begin{equation*}
V(t)-{\tt m}^{-1}t\geq -{\tt m}^{-1}\int_0^t (1-G(y)){\rm d}y\geq -c_1-c_2t^{1-a},\quad t\geq 0.
\end{equation*}
Thus, we have proved the following.
\begin{lemma}
Assume that $\E\xi^2<\infty$. If $\me\eta<\infty$, then
\begin{equation}\label{lord2}
|V(t)-{\tt m}^{-1}t|\leq c_V,\quad t\geq 0
\end{equation}
where $c_V:=\max(c_U, {\tt m}^{-1}\me\eta)$ and  $c_U={\tt m}^{-2} \me\xi^2$. If $\me\eta^a<\infty$ for some $a\in (0,1)$, then
\begin{equation}\label{lord3}
-c_1-c_2t^{1-a}\leq V(t)-{\tt m}^{-1}t\leq c_V,\quad t\geq 0
\end{equation}
for appropriate positive constants $c_1$ and $c_2$.
\end{lemma}

Lemma \ref{aux123} is needed for the proof of Theorem \ref{thm:variance}. Put $\tilde U(t):=\sum_{r\geq 1}\mmp\{S_r\leq t\}=U(t)-1$ for $t\geq 0$.
\begin{lemma}\label{aux123}
Assume that the distribution of $\xi$ is nonlattice and $\me\xi^2<\infty$. Then, for fixed $j\in\mn$,
\begin{equation}\label{inter3}
\int_{[0,\,t]}(t-y)^j{\rm d}\tilde U(y)=\frac{t^{j+1}}{(j+1){\tt m}}+ (b_U-1)t^j +o(t^j),\quad t\to\infty,
\end{equation}
where $b_U:=\me \xi^2/(2{\tt m}^2)$.
\end{lemma}
\begin{proof}
Using the easily checked formula $$\int_{[0,\,t]}(t-y)^j{\rm d}\tilde U(y)=j\int_0^t\int_{[0,\,s]}(s-y)^{j-1}{\rm d}\tilde U(y){\rm d}s,\quad j\in\mn,~~t\geq 0$$ and the mathematical induction one can show that $$\int_{[0,\,t]}(t-y)^j{\rm d}\tilde U(y)=j!\int_0^t\int_0^{y_j}\ldots\int_0^{y_2}\tilde U(y_1){\rm d}y_1\ldots{\rm d}y_j,\quad j\in\mn,~~ t\geq 0.$$ Now \eqref{inter3} follows from the latter equality and the relation $\tilde U(t)={\tt m}^{-1}t+b_U-1+o(1)$ as $t\to\infty$ which is nothing else but formula \eqref{asymp1} with $j=1$ and $\eta=\xi$ a.s.
\end{proof}

Lemma \ref{second} will be used in the proof of Theorem \ref{strong}.
\begin{lemma}\label{second}
Assume that ${\tt m}=\me\xi<\infty$. Then $$\lim_{t\to\infty}\frac{\me (N(t))^2}{t^2}=\frac{1}{{\tt m}^2}.$$
\end{lemma}
\begin{proof}
The relation $${\lim\inf}_{t\to\infty}\frac{\me (N(t))^2}{t^2}\geq \frac{1}{{\tt m}^2}$$ follows from $\me (N(t))^2\geq (V(t))^2$ and Proposition \ref{elem}. The converse inequality for the limit superior is implied by the inequality
$$N(t)\leq \nu(t)=\sum_{i\geq 0}\1_{\{S_i\leq t\}},\quad t\geq 0\quad\text{a.s.}$$ and $\lim_{t\to\infty}t^{-2}\me (\nu(t))^2={\tt m}^{-2}$ (see Theorem 5.1 (ii) on p.~57 in \cite{Gut:2009}).
\end{proof}

Propositions \ref{prConvN} and \ref{suprem} are important ingredients in the proof of Theorem \ref{func2}.
\begin{assertion} \label{prConvN}
Assume that ${\tt m}=\me\xi<\infty$, ${\tt s}^2={\rm Var}\, \xi\in (0,\infty)$ and $\me\eta^a<\infty$ for some $a>0$. Then
\begin{equation} \label{convN}
\Big(\frac{N(ut)-V(ut)}{\sqrt{{\tt m}^{-3}
{\tt s}^2 t}}\Big)_{u\geq 0}~\Rightarrow~(B(u))_{u\geq 0},\quad t\to\infty
\end{equation}
in the $J_1$-topology on $D$, where $(B(u))_{u\geq 0}$ is a standard Brownian motion.
\end{assertion}
\begin{proof}
According to part (B1) of Theorem 3.2 in \cite{Alsmeyer+Iksanov+Marynych:2017},
$$\Big(\frac{N(ut)-{\tt m}^{-1}\int_0^{ut}G(y){\rm d}y}{\sqrt{{\tt m}^{-3}
{\tt s}^2 t}}\Big)_{u\geq 0}~\Rightarrow~(B(u))_{u\geq 0},\quad t\to\infty$$ in the $J_1$-topology on $D$, where, as before, $G$ is the distribution function of $\eta$. Thus, it is enough to show that, for all $T>0$,
\begin{equation}\label{aux231}
\lim_{t\to\infty} t^{-1/2}\sup_{u\in[0,\,T]}\Big|V(ut)-{\tt m}^{-1}\int_0^{ut}G(y){\rm d}y\Big|=0.
\end{equation}
According to \eqref{lord}, for $u\in[0,T]$ and $t\geq 0$,
$$0\leq V(ut)-{\tt m}^{-1}\int_0^{ut} G(y){\rm d}y= \int_{[0,\,ut]}(U(tu-y)-{\tt m}^{-1}(tu-y)){\rm d}G(y)\leq c_U,$$ and \eqref{aux231} follows.
\end{proof}

\begin{assertion}\label{suprem}
Assume that ${\tt m}=\me\xi<\infty$, ${\tt s}^2={\rm Var}\, \xi\in (0,\infty)$ and $\me\eta^a<\infty$ for some $a>0$. Then
\begin{equation*}
\me \sup_{s\in [0,\,t]}(N(s)-V(s))^2=O(t),\quad t\to\infty.
\end{equation*}
\end{assertion}
\begin{proof}
In the case $\me\eta<\infty$ this limit relation is proved in Lemma 4.2(b) of \cite{Gnedin+Iksanov:2020}.

From now on we assume that $\me\eta^a<\infty$ for some $a\in (0,1)$ and $\me\eta=\infty$. As in the proof of Lemma 4.2(b) of \cite{Gnedin+Iksanov:2020} we
shall use a decomposition $$N(t)-V(t)=\sum_{k\geq 0}(\1_{\{S_k+\eta_{k+1}\leq t\}}-G(t-S_k))+\sum_{k\geq 0}G(t-S_k)-V(t),$$ where $G$ is the distribution function of $\eta$.
It suffices to prove that
\begin{equation}\label{11}
\me\Big[\sup_{s\in [0,\,t]}\Big(\sum_{k\geq
0}(\1_{\{S_k+\eta_{k+1}\leq s\}}-G(s-S_k))\Big)^2\Big]=O(t),\quad t\to\infty
\end{equation}
and
\begin{equation}\label{12}
\me\Big[\sup_{s\in [0,\,t]}\Big(\sum_{k\geq 0}G(t-S_k)-V(t)\Big)^2\Big]=O(t),\quad t\to\infty.
\end{equation}
The proof of \eqref{12} given in \cite{Gnedin+Iksanov:2020} goes through for any distribution of $\eta$ and as such applies without changes under the present assumptions.
On the other hand, the proof of \eqref{11} given in \cite{Gnedin+Iksanov:2020} depends crucially on the assumption $\me\eta<\infty$ imposed in Lemma 4.2(b) of the cited article. Thus, another argument has to be found.

For $u,t\geq 0$, put
$$Z_t(u):=\sum_{k\geq 0}\big(\1_{\{S_k+\eta_{k+1}\leq ut\}}-G(ut-S_k))\1_{\{S_k\leq ut\}},$$ so that $$\me\Big[\sup_{s\in [0,\,t]}\Big(\sum_{k\geq 0}(\1_{\{S_k+\eta_{k+1}\leq s\}}-G(s-S_k))\Big)^2\Big]=\me [\sup_{u\in [0,\,1]}(Z_t(u))^2].$$ In what follows we write $\sup_{u\in K}$ when the supremum is taken over an uncountable set $K$ and $\max_{m\leq k\leq n}$ when the maximum is taken over the discrete set $\{m,m+1,\ldots, n\}$.   We start by observing that, for positive integer $I=I(t)$ to be chosen later in \eqref{it},
\begin{multline*}
\sup_{u\in [0,\,1]}|Z_t(u)|=\max_{0\leq k\leq 2^I-1}\sup_{u\in [0,\,2^{-I}]}\big|Z_t(k2^{-I}+u)-Z_t(k2^{-I})+Z_t(k2^{-I})\big|\\\leq \max_{0\leq k\leq 2^I}|Z_t(k2^{-I})|+\max_{0\leq k\leq 2^I-1}\sup_{u\in [0,\,2^{-I}]}|Z_t(k2^{-I}+u)-Z_t(k2^{-I})|.
\end{multline*}
We have used subadditivity of the supremum for the last inequality. We proceed as on p.~764 in \cite{Resnick+Rootzen:2000}. Put $F_j:=\{k2^{-j}:0\leq k\leq 2^j\}$ for $j\in\mn_0$ and fix $u\in F_I$. Now define $u_j:=\max\{w\in F_j: w\leq u\}$ for nonnegative integer $j\leq I$. Then $u_{j-1}=u_j$ or $u_{j-1}=u_j-2^{-j}$. With this at hand, $$|Z_t(u)|=\Big|\sum_{j=1}^I (Z_t(u_j)-Z_t(u_{j-1}))+Z_t(u_0)\Big|\leq \sum_{j=0}^I \max_{1\leq k\leq 2^j}|Z_t(k2^{-j})-Z_t((k-1)2^{-j})|.$$ Combining the fragments together we arrive at the inequality which a starting point of our subsequent work $$\sup_{u\in [0,\,1]}|Z_t(u)|\leq \sum_{j=0}^I \max_{1\leq k\leq 2^j}|Z_t(k2^{-j})-Z_t((k-1)2^{-j})|+\max_{0\leq k\leq 2^I-1}\sup_{u\in [0,\,2^{-I}]}|Z_t(k2^{-I}+u)-Z_t(k2^{-I})|.$$ Thus, \eqref{11} follows if we can show that
\begin{equation}\label{tight1}
\me \bigg(\sum_{j=0}^I \max_{1\leq k\leq 2^j}|Z_t(k2^{-j})-Z_t((k-1)2^{-j})|\bigg)^2=O(t),\quad t\to\infty
\end{equation}
and that
\begin{equation}\label{tight2}
\me [\max_{0\leq k\leq 2^{I}-1}\sup_{u\in [0,\,2^{-I}]}(Z_t(k2^{-I}+u)-Z_t(k2^{-I}))^2]=O(t),\quad t\to\infty.
\end{equation}

We intend to prove that, for $u,v\geq 0$, $u>v$ and $t\geq 0$,
\begin{equation}\label{eq:ineq}
\me (Z_t(u)-Z_t(v))^2\leq 2\me \nu(1) a((u-v)t),
\end{equation}
where, for $t\geq 0$, $a(t):=\sum_{k=0}^{[t]+1}(1-G(k))$ and $\nu(t)=\inf\{k\in\mn: S_k>t\}=\sum_{k\geq 0}\1_{\{S_k\leq t\}}$. Indeed,
\begin{multline*}
\me (Z_t(u)-Z_t(v))^2=\int_{(vt,\,ut]}G(ut-y)(1-G(ut-y)){\rm d}\nu(y)\\+\int_{[0,\,vt]}(G(ut-y)-G(vt-y))(1-G(ut-y)+G(vt-y)){\rm d}\nu(y)\\\leq \int_{(vt,\,ut]}(1-G(ut-y)){\rm
d}\nu(y)+\int_{[0,\,vt]}(G(ut-y)-G(vt-y)){\rm d}\nu(y).
\end{multline*}
Using Lemma \ref{impo1} with
$f(y)=(1-G(y))\1_{[0,\,(u-v)t)}(y)$ and $f(y)=G((u-v)t+y)-G(y)$,
respectively, we obtain
\begin{multline}
\me \int_{(vt,\,ut]}(1-G(ut-y)){\rm d}\nu(y)=\me \int_{[0,\,ut]}(1-G(ut-y))\1_{[0,\,(u-v)t)}(ut-y){\rm d}\nu(y)\\\leq \me \nu(1) \sum_{n=0}^{[ut]}\sup_{y\in
[n,\,n+1)}((1-G(y))\1_{[0,\,(u-v)t)}(y))\\ \leq \me \nu(1) \sum_{n=0}^{[(u-v)t]}(1-G(n))\leq \me \nu(1) a((u-v)t) \label{11a}
\end{multline}
and
\begin{multline}
\me \int_{[0,\,vt]}(G(ut-y)-G(vt-y)){\rm d}\nu(y)\leq \me \nu(1) \sum_{n=0}^{[vt]}\sup_{y\in [n,\,n+1)}(G((u-v)t+y)-G(y))\\ \leq \me \nu(1) \Big(\sum_{n=0}^{[vt]}(1-G(n))-\sum_{n=0}^{[vt]}(1-G((u-v)t+n+1))\Big)\\ \leq \me \nu(1) \Big(\sum_{n=0}^{[vt]}(1-G(n))-\sum_{n=0}^{[ut]+1}(1-G(n))+\sum_{n=0}^{[(u-v)t]+1}(1-G(n))\Big) \\ \leq
\me \nu(1) a((u-v)t).\label{11b}
\end{multline}
Combining \eqref{11a} and \eqref{11b} yields \eqref{eq:ineq}.

\noindent {\sc Proof of \eqref{tight1}}. The assumption $\me \eta^a<\infty$ entails $\lim_{t\to\infty}t^a(1-G(t))=0$ and thereupon, given $C>0$ there exists $t_1>0$ such that $a(t)\leq Ct^{1-a}$ whenever $t\geq t_1$. Using this in combination with \eqref{eq:ineq} yields
\begin{equation}\label{ineq}
\me (Z_t(k2^{-j})-Z_t((k-1)2^{-j}))^2 \leq 2C\me \nu(1) 2^{-j(1-a)}=:C_12^{-j(1-a)}
\end{equation}
whenever $2^{-j}t\geq t_1$. Let $I=I(t)$ denote the integer number
satisfying
\begin{equation}\label{it}
2^{-I}t\geq t_1>2^{-I-1}t.
\end{equation}
Then the inequalities \eqref{ineq} and
\begin{eqnarray*}
\me [(\max_{1\leq k\leq
2^j}(Z_t(k2^{-j})-Z_t((k-1)2^{-j}))^2]&\leq&
\sum_{k=1}^{2^j}\me (Z_t(k2^{-j})-Z_t((k-1)2^{-j}))^2 \\&\leq& C_1 2^{aj}
\end{eqnarray*}
hold whenever $j\leq I$. Invoking the triangle inequality for the $L_2$-norm yields
\begin{multline*}
\me \bigg(\sum_{j=0}^I \max_{1\leq k\leq 2^j}|Z_t(k2^{-j})-Z_t((k-1)2^{-j})|\bigg)^2 \leq \bigg(\sum_{j=0}^I (\me [\max_{1\leq k\leq
2^j}(Z_t(k2^{-j})-Z_t((k-1)2^{-j}))^2]^{1/2}\bigg)^2 \\\leq C_1\bigg(\sum_{j=0}^I 2^{aj/2}\bigg)^2=\big(O\big(2^{aI/2}\big)\big)^2=O(t^a),\quad t\to\infty.
\end{multline*}
Here, the last equality is ensured by the choice of $I$.

\noindent {\sc Proof of \eqref{tight2}}. We shall use a
decomposition
\begin{multline*}
Z_t(k2^{-I}+u)-Z_t(k2^{-I})\\=\sum_{j\geq 0}\big(\1_{\{S_j+\eta_{j+1}\leq
(k2^{-I}+u)t\}}-G((k2^{-I}+u)t-S_j)\big)\1_{\{k2^{-I}t<S_j\leq (k2^{-I}+u)t\}}\\+ \sum_{j\geq 0}\big(\1_{\{k2^{-I}t<S_j+\eta_{j+1}\leq
(k2^{-I}+u)t\}}-(G((k2^{-I}+u)t-S_j)-G(k2^{-I}t-S_j))\big)\1_{\{S_j\leq k2^{-I}t\}}\\=: J_1(t,k,u)+J_2(t,k,u).
\end{multline*}
It suffices to prove that, for $i=1,2$,
\begin{equation}\label{33}
\me [\max_{0\leq k\leq 2^I-1}\sup_{u\in [0,\,2^{-I}]}(J_i(t,k,u))^2]=O(t),\quad t\to\infty.
\end{equation}

\noindent {\sc Proof of \eqref{33} for $i=1$}. Since
$|J_1(t,k,u)|\leq \nu((k2^{-I}+u)t)-\nu(k2^{-I}t)$ and $t\mapsto \nu(t)$ is
a.s.\ nondecreasing we infer $\sup_{u\in [0,\, 2^{-I}]} |J_1(t,k,u)|\leq \nu ((k+1)2^{-I}t)-\nu(k2^{-I}t)$ a.s. Hence,
\begin{multline*}
\me [\max_{0\leq k\leq 2^I-1}\sup_{u\in [0,\,2^{-I}]}(J_i(t,k,u))^2]\leq \me [\max_{0\leq k\leq 2^I-1}(\nu((k+1)2^{-I}t)-\nu(k2^{-I}t))^2]\\ \leq \me \sum_{k=0}^{2^I-1} \me (\nu((k+1)2^{-I}t)-\nu(k2^{-I}t))^2\leq 2^I \me (\nu(2^{-I}t))^2\leq (t/t_1)\me (\nu(2t_1))^2=O(t),\quad t\to\infty.
\end{multline*}
Here, the second inequality follows from distributional subadditivity of $\nu(t)$ (see, for instance, formula (5.7) on p.~58 in \cite{Gut:2009}), and the third inequality is secured by the choice of $I$.

\noindent {\sc Proof of \eqref{33} for $i=2$}. We have
\begin{multline*}
\sup_{u\in [0,\, 2^{-I}]} |J_2(t,k,u)|\leq \sup_{u\in [0,\, 2^{-I}]} \bigg(\sum_{j\geq 0}\1_{\{k2^{-I}t<S_j+\eta_{j+1}\leq
(k2^{-I}+u)t\}}\1_{\{S_j\leq k2^{-I}t\}}\\+\sum_{j\geq
0}(G((k2^{-I}+u)t-S_j)-G(k2^{-I}t-S_j))\big)\1_{\{S_j\leq k2^{-I}t\}}\bigg)\\\leq \sum_{j\geq
0}\1_{\{k2^{-I}t<S_j+\eta_{j+1}\leq (k+1)2^{-I}t\}}\1_{\{S_j\leq
k2^{-I}t\}}\\+\sum_{j\geq 0}(G(((k+1)2^{-I})t-S_j)-G(k2^{-I}t-S_j))\1_{\{S_j\leq k2^{-I}t\}}\\\leq \Big|\sum_{j\geq 0}\big(\1_{\{k2^{-I}t<S_j+\eta_{j+1}\leq
(k+1)2^{-I}t\}}-(G(((k+1)2^{-I})t-S_j)-G(k2^{-I}t-S_j))\big)\1_{\{S_j\leq k2^{-I}t\}}\Big|\\+ 2\sum_{j\geq 0}(G(((k+1)2^{-I})t-S_j)-G(k2^{-I}t-S_j))\big)\1_{\{S_j\leq k2^{-I}t\}}\\=:J_{21}(t,k)+2J_{22}(t,k).
\end{multline*}

Using \eqref{11b} with $u=(k+1)2^{-I}$ and $v=k2^{-I}$ we obtain $$\me
(J_{22}(t,k))^2 \leq \me (\nu(1))^2 (a(2^{-I}t))^2 \leq \me
(\nu(1))^2 (a(2t_1))^2$$ which implies
\begin{eqnarray*}
\me (\max_{0\leq k\leq 2^{I}-1}J_{22}(t,k))^2 &\leq& 2^I \max_{0\leq k\leq 2^I-1}\me
(J_{22}(t,k))^2 \\&\leq& (t/t_1) \me (\nu(1))^2 (a(2t_1))^2=O(t),\quad t\to\infty.
\end{eqnarray*}
Further, by \eqref{11b},
\begin{equation*}
\me (J_{21}(t,k))^2= \me \int_{[0,\,k2^{-I}t]}(G((k+1)2^{-I}t-y)-G(k2^{-I}t-y)){\rm d}\nu(y)\leq \me \nu(1) a(2^{-I}t).
\end{equation*}
Hence, $\me (\max_{0\leq k\leq 2^{I}-1}J_{21}(t,k))^2=O(t)$ by the same reasoning as above, and \eqref{33} for $i=2$ follows. The proof of Proposition \ref{suprem} is complete.
\end{proof}

\section{Proofs of the main results}\label{pro}

\subsection{Proofs of Proposition \ref{elem} and Theorem \ref{rate}}

\begin{proof}[Proof of Proposition \ref{elem}]
The simplest way to prove this is to use Laplace transforms. Indeed, for fixed $j\in\mn$, $$\int_{[0,\,\infty)}e^{-st}{\rm d}V_j(t)=\Big(\frac{\me e^{-s\eta}}{1-\me e^{-s\xi}}\Big)^j~\sim~ \frac{1}{{\tt m}^js^j},\quad s\to 0+.$$ By Karamata's Tauberian theorem (Theorem 1.7.1 in \cite{Bingham+Goldie+Teugels:1989}), \eqref{elem1} holds.
\end{proof}

\begin{proof}[Proof of Theorem \ref{rate}]
We use induction on $j$. Let $j=1$. Write
\begin{equation}\label{I}
V(t)-{\tt m}^{-1}t=\int_{[0,\,t]}(U(t-y)-{\tt m}^{-1}(t-y)){\rm d}G(y)-{\tt m}^{-1}\int_0^t (1-G(y)){\rm d}y.
\end{equation}
Plainly, the second term converges to $-{\tt m}^{-1}\me \eta$ as $t\to\infty$. It is a simple consequence of the Blackwell theorem that $$\lim_{t\to\infty}(U(t)-{\tt m}^{-1}t)= (2 {\tt m}^2)^{-1}\me \xi^2=b_U.$$ Using this in combination with \eqref{lord2} we invoke Lebesgue's dominated convergence theorem to infer that the first term in \eqref{I} converges to $b_U$ as $t\to\infty$. Thus, we have shown that \eqref{asymp1} with $j=1$ holds true.

Assume that \eqref{asymp1} holds for $j=k$. In view of \eqref{asymp1} with $j=1$, given $\varepsilon>0$ there exists $t_0>0$ such that
\begin{equation}\label{bound}
|V(t)-{\tt m}^{-1}t-b_V|\leq \varepsilon
\end{equation}
whenever $t\geq t_0$. Write, for $t\geq t_0$,
\begin{multline*}
V_{k+1}(t)-\frac{t^{k+1}}{(k+1)!{\tt m}^{k+1}}=\int_{[0,\,t-t_0]}(V(t-y)-{\tt m}^{-1}(t-y)){\rm d}V_k(y)\\+\int_{(t-t_0,\,t]}(V(t-y)-{\tt m}^{-1}(t-y)){\rm d}V_k(y)+{\tt m}^{-1}\int_0^t\Big(V_k(y)-\frac{y^k}{k!{\tt m}^k}\Big){\rm d}y=I_1(t)+I_2(t)+I_3(t).
\end{multline*}
In view of \eqref{bound}, $$(b_V-\varepsilon)V_k(t-t_0)\leq I_1(t)\leq (b_V+\varepsilon)V_k(t-t_0),$$ whence $$\frac{b_V-\varepsilon}{k!{\tt m}^k}\leq {\lim\inf}_{t\to\infty}\frac{I_1(t)}{t^k}\leq {\lim\sup}_{t\to\infty}\frac{I_1(t)}{t^k}\leq \frac{b_V+\varepsilon}{k!{\tt m}^k}$$ by Proposition \ref{elem}.

Using \eqref{lord2} we obtain $|I_2(t)|\leq c_V (V_k(t)-V_k(t-t_0))$ for all $t\geq t_0$, whence $\lim_{t\to\infty}t^{-k}I_2(t)=0$ by Theorem \ref{blackwell_early}. A combination of this with the last centered formula and sending $\varepsilon\to 0+$ we infer $$I_1(t)+I_2(t)~\sim~\frac{b_V t^k}{k!{\tt m}^k},\quad t\to\infty.$$ Finally, by the induction assumption and  L'H\^{o}pital's rule $$I_3(t)\sim \frac{b_V t^k}{(k-1)!{\tt m}^k},\quad t\to\infty.$$ Combining fragments together we arrive at \eqref{asymp1} with $j=k+1$.
\end{proof}

\subsection{Proofs of Theorems \ref{keyren} and \ref{blackwell_early}}\label{pro1}

\begin{proof}[Proof of Theorem \ref{keyren}]
The proof of Theorem 2.7 (a) in \cite{Bohun etal:2021} applies, with obvious simplifications. Note that for early generations ($j$ is fixed), asymptotic relation \eqref{elem1} holds under the sole assumption ${\tt m}<\infty$. This is not the case for intermediate generations ($j=j(t)\to\infty$, $j(t)=o(t)$ as $t\to\infty$) treated in \cite{Bohun etal:2021} which explains the appearance of the additional assumption $\me\xi^r<\infty$ for some $r\in (1,2]$ in Theorem 2.7 of the cited paper.
\end{proof}

\begin{proof}[Proof of Theorem \ref{blackwell_early}]
When $j=1$, relation \eqref{bla_early} holds by Lemma 4.2 (a) in \cite{Bohun etal:2021}.
Write
\begin{multline*}
V_j(t+h)-V_j(t)=\int_{[0,\,t]}(V(t+h-y)-V(t-y)){\rm d}V_{j-1}(y)+\int_{(t,\,t+h]}V(t+h-y){\rm d}V_{j-1}(y)\\=:A_j(t)+B_j(t).
\end{multline*}
We first show that the contribution of $B_j(t)$ is negligible. Indeed, using monotonicity of $V$ and $\lim_{t\to\infty}(V_j(t+h)/V_j(t))=1$ (see Proposition \ref{elem}) we obtain $$B_j(t)\leq V(h)(V_{j-1}(t+h)-V_{j-1}(t))=o(V_{j-1}(t)),\quad t\to\infty.$$ In view of \eqref{bla_early} with $j=1$, given $\varepsilon>0$ there exists $t_0>0$ such that $$|V(t+h)-V(t)-{\tt m}^{-1}h|\leq \varepsilon$$ whenever $t\geq t_0$. Thus, % With this $t_0$,
we have, for $t\geq t_0$, $$A_j(t)=\int_{[0,\,t-t_0]}(V(t+h-y)-V(t-y)){\rm d}V_{j-1}(y)+\int_{(t-t_0,\,t]}V(t+h-y){\rm d}V_{j-1}(y)=:A_{j,1}(t)+A_{j,2}(t).$$ By the argument used for $B_j(t)$ we infer $A_{j,2}(t)=o(V_{j-1}(t))$. Further, $A_{j,1}(t)\leq ({\tt m}^{-1}h+\varepsilon)V_{j-1}(t-t_0)$, whence $$\lim\sup_{t\to\infty}(A_j(t)/V_{j-1}(t))\leq {\tt m}^{-1}h.$$ A symmetric argument proves the converse inequality for the limit inferior. Invoking Proposition \ref{elem} completes the proof of Theorem \ref{blackwell_early}.
\end{proof}

\subsection{Proof of Theorem \ref{thm:variance}}

For $j\in\mn$ and $t\geq 0$, put $D_j(t):={\rm Var}\,N_j(t)$. Recall that, as a consequence of the assumption $\eta=\xi$ a.s., $T_r=S_r$ for $r\in\mn$ and $V(t)=\tilde U(t)=\sum_{r\geq 1}\mmp\{S_r\leq t\}$ for $t\geq 0$. However, we prefer to write $V_j$ rather than $\tilde U_j$.

Let $j\in\mn$, $j\geq 2$. We obtain with the help of \eqref{basic1232}
\begin{eqnarray}
N_j(t)-V_j(t)&=&\sum_{r\geq 1}\big(N^{(r)}_{j-1}(t-S_r)-V_{j-1}(t-S_r)\big)\1_{\{S_r\leq t\}}\notag\\&+& \bigg(\sum_{r\geq 1}V_{j-1}(t-S_r)\1_{\{S_r\leq t\}}-V_j(t)\bigg)=:N_{j,\,1}(t)+N_{j,\,2}(t),\quad t\geq 0,\label{defNj}
\end{eqnarray}
whence
\begin{equation}\label{aux5}
D_j(t)=\E (N_{j,\,1}(t))^2+\E (N_{j,\,2}(t))^2.
\end{equation}
We start by showing that the following asymptotic relations hold, for $j\geq 2$, as $t\to\infty$,
\begin{multline}\label{aux3}
\E (N_{j,\,2}(t))^2={\rm Var}\,\Big(\sum_{r\geq 1}V_{j-1}(t-S_r)\1_{\{S_r\leq t\}}\Big)\\=\E\Big(\sum_{r\geq 1}V_{j-1}(t-S_r)\1_{\{S_r\leq t\}}\Big)^2
-V_j^2(t)~\sim~\frac{{\tt s}^2}{(2j-1)((j-1)!)^2 {\tt m}^{2j+1}}t^{2j-1}, \quad t\to\infty.
\end{multline}

\noindent {\sc Proof of \eqref{aux3}}. We shall use the equality (see formula (4.9) in \cite{Iksanov+Kabluchko:2018})
\begin{equation}\label{mom}
\E\bigg(\sum_{r\geq 1}V_{j-1}(t-S_r)\1_{\{S_r\leq t\}}\bigg)^2 = 2\int_{[0,\,t]}V_{j-1}(t-y)V_j(t-y){\rm
d}\tilde U(y)+\int_{[0,\,t]}V_{j-1}^2(t-y){\rm d}\tilde U(y).
\end{equation}
In view of \eqref{asymp1}
\begin{multline*}
\int_{[0,\,t]}V_{j-1}(t-y)V_j(t-y){\rm d}\tilde U(y)=\int_{[0,\,t]}\Big(\frac{(t-y)^{j-1}}{(j-1)!{\tt m}^{j-1}}+\frac{b_V(j-1)(t-y)^{j-2}}{(j-2)!{\tt m}^{j-2}}+o((t-y)^{j-2})\Big)\\\times \Big(\frac{(t-y)^j}{j!{\tt m}^j}+\frac{b_Vj(t-y)^{j-1}}{(j-1)!{\tt m}^{j-1}}+o((t-y)^{j-1})\Big){\rm d}\tilde U(y)=\frac{1}{(j-1)!j!{\tt m}^{2j-1}}\int_{[0,\,t]}(t-y)^{2j-1}{\rm d}\tilde U(y)\\+\frac{b_V(2j^2-2j+1)}{(j-1)!j!{\tt m}^{2j-2}}\int_{[0,\,t]}(t-y)^{2j-2}{\rm d}\tilde U(y)+\int_{[0,\,t]}o((t-y)^{2j-2}){\rm d}\tilde U(y),
\end{multline*}
where $b_V=\me \xi^2/(2{\tt m}^2)-1$ because under the present assumption $\me \eta={\tt m}$. According to \eqref{inter3}
\begin{equation*}
\int_{[0,\,t]}(t-y)^{2j-1} {\rm d}\tilde U(y)=\frac{t^{2j}}{2j{\tt m}}+b_V t^{2j-1}+o(t^{2j-1}),\quad t\to\infty;
\end{equation*}
\begin{equation*}
\int_{[0,\,t]}(t-y)^{2j-2} {\rm d}\tilde U(y)=\frac{t^{2j-1}}{(2j-1){\tt m}}+o(t^{2j-1}),\quad t\to\infty.
\end{equation*}
Also, it can be checked that
\begin{equation*}
\int_{[0,\,t]}o((t-y)^{2j-2}){\rm d}\tilde U(y)=o(t^{2j-1}),\quad t\to\infty.
\end{equation*}
By Proposition \ref{elem}, $$V_{j-1}^2(t)~\sim~ \frac{t^{2j-2}}{((j-1)!)^2 {\tt m}^{2j-2}}\quad\text{and}\quad V(t)=\tilde U(t)~\sim~\frac{t}{{\tt m}},\quad t\to\infty.$$ As in the proof of Proposition \ref{elem} we now invoke Karamata's Tauberian theorem (Theorem 1.7.1 in \cite{Bingham+Goldie+Teugels:1989}) to obtain $$\int_{[0,\,t]}V_{j-1}^2(t-y){\rm d}\tilde U(y)~\sim~\frac{t^{2j-1}}{(2j-1)((j-1)!)^2{\tt m}^{2j-1}},\quad t\to\infty.$$ Using the aforementioned asymptotic relations and recalling \eqref{mom} we infer
\begin{multline*}
\E\bigg(\sum_{r\geq 1}V_{j-1}(t-S_r)\1_{\{S_r\leq t\}}\bigg)^2=\frac{t^{2j}}{(j!)^2{\tt m}^{2j}}+\frac{2b_Vt^{2j-1}}{(j-1)!j!{\tt m}^{2j-1}}+\frac{2b_V(2j^2-2j+1)t^{2j-1}}{(2j-1)(j-1)!j!{\tt m}^{2j-1}}\\+\frac{t^{2j-1}}{(2j-1)((j-1)!)^2{\tt m}^{2j-1}}+o(t^{2j-1}),\quad t\to\infty.
\end{multline*}
Further, as $t\to\infty$,
$$V_j^2(t)=\frac{t^{2j}}{(j!)^2{\tt m}^{2j}}+\frac{2t^j}{j!{\tt m}^j}\bigg(V_j(t)-\frac{t^j}{j!{\tt m}^j}\bigg)+\bigg(V_j(t)-\frac{t^j}{j!{\tt m}^j}\bigg)^2=
\frac{t^{2j}}{(j!)^2{\tt m}^{2j}}+\frac{2b_Vt^{2j-1}}{((j-1)!)^2 {\tt m}^{2j-1}}+ o(t^{2j-1})$$
having utilized \eqref{asymp1}. The last two asymptotic relations
entail \eqref{aux3}.

With \eqref{aux3} at hand we are ready to prove \eqref{ineq3}. To this end, we shall use the mathematical induction. If
$j=1$, \eqref{ineq3} takes the form $D_1(t)=\E (N(t)-\tilde U(t))^2\sim {\tt s}^2 {\tt m}^{-3} t$ as $t\to\infty$. This can be checked along the lines of the proof of \eqref{aux3}. Alternatively, this relation follows from Theorem 3.8.4 in \cite{Gut:2009} where the assumption that the distribution of $\xi$ is nonlattice is not made. Assume that \eqref{ineq3} holds for $j=k-1\geq 1$, that is,
\begin{equation*}
D_{k-1}(t)~\sim~ \frac{{\tt s}^2}{(2k-3)((k-2)!)^2 {\tt m}^{2k-1}}t^{2k-3},\quad t\to\infty.
\end{equation*}
Using this and the equality
\begin{equation*}
\E (N_{k,1}(t))^2=\E\sum_{r\geq 1}D_{k-1}(t-S_r)\1_{\{S_r\leq t\}}=\int_{[0,\,t]}D_{k-1}(t-y){\rm d}\tilde U(y)
\end{equation*}
in combination with Karamata's Tauberian theorem (Theorem 1.7.1 in \cite{Bingham+Goldie+Teugels:1989}) or, even simpler, a bare hands calculation we infer
\begin{equation*}
\E (N_{k,\,1}(t))^2~\sim~\frac{{\tt s}^2}{(2k-3)(2k-2)((k-2)!)^2 {\tt m}^{2k}} t^{2k-2},\quad t\to\infty.
\end{equation*}
By virtue of \eqref{aux5} and \eqref{aux3} we conclude that
\eqref{ineq3} holds for $j=k$. The proof of Theorem \ref{thm:variance} is complete.

\subsection{Proof of Theorem \ref{strong}}
We use the mathematical induction. When $j=1$, \eqref{slln} holds true by formula (24) in \cite{Alsmeyer+Iksanov+Marynych:2017}. Assuming that it holds for $j=k$ we intend to show that \eqref{slln} also holds for $j=k+1$. To this end, we write with the help of \eqref{basic1232} for $j=k+1$:
\begin{equation}\label{x5}
N_{k+1}(t)=\sum_{r\geq 1}\big(N^{(r)}_{1,\,k+1}(t-T_r^{(k)})-V(t-T_r^{(k)})\big)\1_{\{T_r^{(k)}\leq t\}}+\int_{[0,\,t]}V(t-y){\rm d}N_k(y),\quad t\geq 0.
\end{equation}
By Proposition \ref{elem}, given $\varepsilon>0$ there exists $t_0>0$ such that $|t^{-1} V(t)-{\tt m}^{-1}|\leq \varepsilon$ whenever $t\geq t_0$. We have $$\int_{(t-t_0,\,t]}V(t-y){\rm d}N_k(y)\leq V(t_0)(N_k(t)-N_k(t-t_0))=o(t^k)\quad \text{a.s.~~as}\quad t\to\infty$$ by the induction assumption. Analogously, $$\int_{(t-t_0,\,t]}(t-y) {\rm d}N_k(y)=o(t^k)\quad \text{a.s.~~as}\quad t\to\infty.$$
Further,
\begin{multline*}
\int_{[0,\,t-t_0]}V(t-y){\rm d}N_k(y)\geq ({\tt m}^{-1}-\varepsilon)\int_{[0,\,t-t_0]}(t-y){\rm d}N_k(y)\geq ({\tt m}^{-1}-\varepsilon)\Big(\int_{[0,\,t]}(t-y){\rm d}N_k(y)\\-\int_{(t-t_0,\,t]}(t-y){\rm d}N_k(y)\Big).
\end{multline*}
Using $\int_{[0,\,t]}(t-y){\rm d}N_k(y)=\int_0^t N_k(y){\rm d}y$ and applying L'H\^{o}pital's rule in combination with the induction assumption we infer $$\lim_{t\to\infty}\frac{\int_{[0,\,t]}(t-y){\rm d}N_k(y)}{t^{k+1}}=\frac{1}{{\tt m}^k(k+1)!}\quad\text{a.s.}$$ Combining pieces together we arrive at $${\lim\inf}_{t\to\infty}\frac{\int_{[0,\,t]}V(t-y){\rm d}N_k(y)}{t^{k+1}}\geq \frac{1}{{\tt m}^{k+1}(k+1)!}\quad\text{a.s.}$$ The converse inequality for the limit superior follows similarly, whence
\begin{equation}\label{aux234}
\lim_{t\to\infty}\frac{\int_{[0,\,t]}V(t-y){\rm d}N_k(y)}{t^{k+1}}=\frac{1}{{\tt m}^{k+1}(k+1)!}\quad\text{a.s.}
\end{equation}

Further,
\begin{multline*}
\me \Big(\sum_{r\geq 1}\big(N^{(r)}_{1,\,k+1}(t-T_r^{(k)})-V(t-T_r^{(k)})\big)\1_{\{T_r^{(k)}\leq t\}}\Big)^2=\sum_{r\geq 1}\me \big(N^{(r)}_{1,\,k+1}(t-T_r^{(k)})-V(t-T_r^{(k)})\big)^2\1_{\{T_r^{(k)}\leq t\}}\\\leq \int_{[0,\,t]}\me (N(t-y))^2{\rm d}V_k(y)\leq \me (N(t))^2 V_k(t)=O(t^{k+2}),\quad t\to\infty
\end{multline*}
having utilized monotonicity of $t\mapsto \me (N(t))^2$ for the last inequality and Lemmas \ref{elem} and \ref{second} for the last equality. Invoking now the Markov inequality and the Borel-Cantelli lemma we conclude that
\begin{equation}\label{along}
\lim_{n\to\infty}\frac{\sum_{r\geq 1}\big(N^{(r)}_{1,\,k+1}(n^2-T_r^{(k)})-V(n^2-T_r^{(k)})\big)\1_{\{T_r^{(k)}\leq n^2\}}}{n^{2(k+1)}}=0\quad\text{a.s.}
\end{equation}
($n$ approaches $\infty$ along integers). This together with \eqref{aux234} yields
$$\lim_{n\to\infty}\frac{N_{k+1}(n^2)}{n^{2(k+1)}}=\frac{1}{{\tt m}^{k+1}(k+1)!}\quad\text{a.s.}$$ Thus, it remains to show that we may pass to the limit continuously. To this end, note that, for each $t\ge 0$, there exists $n\in\mn$ such that
$t\in [(n-1)^2, n^2)$ and use a.s.\ monotonicity of
$N_{k+1}$ to obtain
$$\frac{(n-1)^{2(k+1)}}{n^{2(k+1)}}\frac{N_{k+1}((n-1)^2)}{(n-1)^{2(k+1)}}\leq \frac{N_{k+1}(t)}{t^{k+1}}\leq \frac{N_{k+1}(n^2)}{n^{2(k+1)}}\frac{n^{2(k+1)}}{(n-1)^{2(k+1)}}\quad\text{a.s.}$$
Letting $t$ tend to $\infty$ we arrive at $$\lim_{t\to\infty}\frac{N_{k+1}(t)}{t^{k+1}}=\frac{1}{{\tt m}^{k+1}(k+1)!}\quad\text{a.s.},$$ thereby completing the induction step. The proof of Theorem \ref{strong} is complete.

\subsection{Proof of Theorem \ref{func2}}

In the case $\me \eta<\infty$ this result follows from Theorem 3.2 and Lemma 4.2 in \cite{Gnedin+Iksanov:2020}. Thus, we concentrate on the case $\me \eta^a<\infty$ for $a\in (0,1)$ and $\me\eta=\infty$.

We are going to apply Theorem \ref{main0} with $N^\ast_j=N_j$ for $j\in\mn$. According to \eqref{lord3},
\begin{equation*}
-c_1-c_2 t^{1-a}\leq V(t)-{\tt m}^{-1}t\leq c_U,\quad t\geq 0
\end{equation*}
for some positive constants $c_1$ and $c_2$ and $c_U={\tt m}^{-2} \me\xi^2$, that is, condition \eqref{1} holds with $c={\tt m}^{-1}$, $\omega=1$, $\varepsilon_1=a$, $\varepsilon_2=1$, $a_0+a_1=c_U$, $b_0=-c_1$, $b_1=-c_2$.
By Proposition \ref{suprem},
\begin{equation*}
\me \sup_{s\in [0,\,t]}(N(s)-V(s))^2=O(t),\quad t\to\infty,
\end{equation*}
that is, condition \eqref{vari} holds with $\gamma=1/2$. By Proposition \ref{prConvN},
\begin{equation*}
\Big(\frac{N(ut)-V(ut)}{\sqrt{{\tt m}^{-3}
{\tt s}^2 t}}\Big)_{u\geq 0}~\Rightarrow~(B(u))_{u\geq 0},\quad t\to\infty
\end{equation*}
in the $J_1$-topology on $D$. This means that condition \eqref{flt_assumption} holds with $\gamma=1/2$, $b={\tt m}^{-3/2}
{\tt s}$ and $W=B$, a Brownian motion. Recall that the process $B$ is locally H\"{o}lder continuous with exponent $\beta$ for any $\beta\in (0,1/2)$. Thus, by Theorem \ref{main0}, relation \eqref{relFunc2} is a specialization of \eqref{relation_main} with $\gamma=1/2$, $\omega=1$, $R^{(1)}_j=B_{j-1}$, $j\in\mn$ and $\rho_j=1/({\tt m}^j j!)$, $j\in\mn_0$.

Now we prove the claim that the centering $V_j(ut)$ can be replaced with that given in \eqref{center}. We first note that the equality in \eqref{center} follows with the help of the mathematical induction in $k$ from the representation
\begin{multline*}
\me (t-R_i)^k\1_{\{R_i\leq t\}}=\int_{[0,\,t]}(t-y)^k {\rm d}\mmp\{R_i\leq y\}\\=k\int_0^t \int_{[0,\,s]}(s-y)^{k-1}{\rm d}\mmp\{R_i\leq y\}{\rm d}s,\quad i,k\in\mn,~~t\geq 0,
\end{multline*}
where $R_i:=\eta_1+\ldots+\eta_i$. Here, the first step of induction is justified by the equality $$\int_{[0,\,t]}(t-y) {\rm d}\mmp\{R_i\leq y\}=\int_0^t \mmp\{R_i\leq s\}{\rm d}s,\quad i\in\mn,~t\geq 0.$$
Further, we show that whenever $\me\xi^2<\infty$ irrespective of the distribution of $\eta$, for all $T>0$,
\begin{equation}\label{aux4}
\lim_{t\to\infty}t^{-(j-1/2)}\sup_{u\in [0,\,T]}|V_j(ut)-(j!{\tt m}^j)^{-1}\me (ut-R_j)^j\1_{\{R_j\leq ut\}}|=0.
\end{equation}
%where $R_j:=\eta_1+\ldots+\eta_j$.
To this end, we recall that, according to formula (4.4) in \cite{Buraczewski+Dovgay+Iksanov:2020} (we use the formula with $\eta=0$), $$U_j(t)-\frac{t^j}{j!{\tt m}^j}\leq \sum_{i=0}^{j-1}\binom{j}{i}\frac{c_U^{j-i}t^i}{i!{\tt m}^i},\quad t\geq 0,$$ where $c_U={\tt m}^{-2} \me\xi^2$. Using this and Lemma \ref{lefthand}, we obtain, for $u\in [0,T]$ and $t\geq 0$,
\begin{multline*}
\Big|V_j(ut)-\frac{\me (ut-R_j)^j\1_{\{R_j\leq ut\}}}{j!{\tt m}^j}\Big|=\int_{[0,\,ut]}\Big(U_j(ut-y)-\frac{(ut-y)^j}{j!{\tt m}^j}\Big){\rm d}\mmp\{R_j\leq y\}\\\leq \int_{[0,\,ut]}\sum_{i=0}^{j-1}\binom{j}{i}\frac{c_U^{j-i}(ut-y)^i}{i!{\tt m}^i}{\rm d}\mmp\{R_j\leq y\}\leq \sum_{i=0}^{j-1}\binom{j}{i}\frac{c_U^{j-i}(Tt)^i}{i!{\tt m}^i}=o(t^{j-1/2}),\quad t\to\infty
\end{multline*}
which proves \eqref{aux4}.

It remains to show that if $\me \eta^{1/2}<\infty$, then the centering $V_j(ut)$ can be replaced with $(ut)^j/(j! {\tt m}^j)$. To justify this, it suffices to check that
\begin{equation}\label{aux4444}
\lim_{t\to\infty}\frac{\sup_{u\in [0,\,T]}\big((ut)^j- j! \int_0^{ut}\int_0^{t_2}\ldots\int_0^{t_j}\mmp\{R_j\leq y\}{\rm d}y{\rm d}t_j\ldots{\rm d}t_2\big)}{t^{j-1/2}}=0.
\end{equation}
The numerator of the ratio under the limit on the left-hand side of \eqref{aux4444} is equal to
$$\sup_{u\in[0,\,T]} \int_0^{tu}\int_0^{t_2}\ldots\int_0^{t_j}\mmp\{R_j>y\}{\rm d}y{\rm d}t_j\ldots{\rm d}t_2= \int_0^{tT}\int_0^{t_2}\ldots\int_0^{t_j}\mmp\{R_j>y\}{\rm d}y{\rm d}t_j\ldots{\rm d}t_2.$$ Hence, we are left with showing that $$\lim_{t\to\infty}t^{-(j-1/2)}\int_0^t\int_0^{t_2}\ldots\int_0^{t_j}\mmp\{R_j>y\}{\rm d}y{\rm d}t_j\ldots{\rm d}t_2=0.$$
Assume that $\me\eta<\infty$, so that $\me R_j<\infty$ and thereupon $\lim_{t\to\infty} \int_0^t\mmp\{R_j>y\}{\rm d}y=\me R_j<\infty$. Then using L'Hospital's rule $(j-1)$-times we obtain $$\lim_{t\to\infty}\frac{\int_0^t\int_0^{t_2}\ldots\int_0^{t_j}\mmp\{R_j>y\}{\rm d}y{\rm d}t_j\ldots{\rm d}t_2}{t^{j-1/2}}=\frac{2^{j-1}}{1\cdot 3\cdot\ldots\cdot(2j-1)} \lim_{t\to\infty}\frac{\int_0^t\mmp\{R_j>y\}{\rm d}y}{t^{1/2}}=0.$$ Assume that $\me\eta=\infty$. Since $\me \eta^{1/2}<\infty$ is equivalent to $\me R_j^{1/2}<\infty$ we infer $\lim_{t\to\infty}t^{1/2}\mmp\{R_j>t\}=0$. With this at hand, using L'Hospital's rule $j$-times we infer $$\lim_{t\to\infty}\frac{\int_0^t\int_0^{t_2}\ldots\int_0^{t_j}\mmp\{R_j>y\}{\rm d}y{\rm d}t_j\ldots{\rm d}t_2}{t^{j-1/2}}=\frac{2^j}{1\cdot 3\cdot\ldots\cdot(2j-1)} \lim_{t\to\infty}t^{1/2}\mmp\{R_j>t\}=0.$$
The proof of Theorem \ref{func2} is complete.

\section{Appendix}

In this section we state several results borrowed from other sources. The first of these can be found in the proof of Lemma 7.3 in
\cite{Alsmeyer+Iksanov+Marynych:2017}.
\begin{lemma}\label{impo1}
Let $f:[0,\infty)\to [0,\infty)$ be a locally bounded function. Then, for any $l\in\mn$,
\begin{equation}\label{impo1-gen}
\E \bigg(\sum_{k\geq 0}f(t-S_k)\1_{\{S_k\leq t\}}\bigg)^l\leq
\bigg(\sum_{j=0}^{[t]}\sup_{y\in[j,\,j+1)}f(y)\bigg)^l\me
(\nu(1))^l,\quad t\geq 0,
\end{equation}
where $\nu(t)=\inf\{k\in\mn: S_k>t\}$.
\end{lemma}

For $j\in\mn$ and $t\geq 0$, denote by $N^\ast_j(t)$ the number of the $j$th generation individuals with birth times $\leq t$ in a general branching process generated by an arbitrary locally finite point process $T^\ast$, and put $V^\ast_j(t):=\me N^\ast_j(t)$. In particular, $N^\ast_j=N_j$ for $j\in\mn$ when $T^\ast=T$. For notational simplicity, put $N^\ast:=N^\ast_1$ and $V^\ast:=V_1^\ast$.

Let $W:=(W(s))_{s\geq 0}$ denote a centered Gaussian process which is
a.s.\ locally H\"{o}lder continuous and
satisfy $W(0)=0$. For each $u>0$, put
$$R^{(u)}_1(s):=W(s),\quad R^{(u)}_j(s):=\int_{[0,\,s]}(s-y)^{u(j-1)}{\rm
d}W(y),\quad s\geq 0,~ j\geq 2.$$ The following result follows from Theorem 3.2 in \cite{Gnedin+Iksanov:2020} and its proof.
\begin{thm}\label{main0}
Assume the following conditions hold:
\begin{itemize}
\item[\rm(i)]
\begin{equation}\label{1}
b_0+b_1 t^{\omega-\varepsilon_1}\leq V^\ast(t)-c t^\omega\leq
a_0+a_1t^{\omega-\varepsilon_2}
\end{equation}
for all $t\geq 0$ and some constants $c,\omega, a_0, a_1>0$, $0<\varepsilon_1,
\varepsilon_2\leq \omega$ and $b_0, b_1\in\mr$,
\item[\rm(ii)]
\begin{equation}\label{vari}
\me \sup_{s\in [0,\,t]}(N^\ast(s)-V^\ast(s))^2=O(t^{2\gamma}),\quad
t\to\infty
\end{equation}
for some $\gamma\in (0, \omega)$.
\item[\rm(iii)]
\begin{equation}\label{flt_assumption}
\Big(\frac{N^\ast(ut)-V^\ast(ut)}{bt^\gamma}\Big)_{u\geq 0}~\Rightarrow~
(W(u))_{u\geq 0},\quad t\to\infty
\end{equation}
in the $J_1$-topology on $D$ for some $b>0$ and the same $\gamma$
as in \eqref{vari}.
\end{itemize}
Then
\begin{equation}\label{relation_main}
\bigg(\Big(\frac{N^\ast_j(ut)-V^\ast_j(ut)}{b\rho_{j-1}
t^{\gamma+\omega(j-1)}}\Big)_{u\geq 0}\bigg)_{j\in\mn}~\Rightarrow~\big(\big(R^{(\omega)}_j(u)\big)_{u\geq 0}\big)_{j\in\mn}
\end{equation}
in the $J_1$-topology on $D^\mn$, where
\begin{equation*}%\label{defcj}
\rho_j:=\frac{(c\Gamma(\omega+1))^j}{\Gamma(\omega j+1)},\quad j\in\mn_0
\end{equation*}
with $\Gamma(\cdot)$ denoting the gamma function.
\end{thm}

\vspace{5mm}

\noindent {\bf Acknowledgement}. The present work was supported by the National Research Foundation of Ukraine (project 2020.02/0014 ``Asymptotic regimes of perturbed random walks: on the edge of modern and classical probability'').

\end{document}